\documentclass[a4paper,11pt,,reqno]{amsart}
\usepackage{amsmath, amsthm, amscd, amsfonts, amssymb, graphicx, color}
\usepackage[bookmarksnumbered, colorlinks, plainpages]{hyperref}
\usepackage{wrapfig}
\usepackage{subfig}
\usepackage{subfig}
\usepackage{cite}
\usepackage[framemethod=default]{mdframed}
\usepackage{showexpl}

\begin{document}
	
	 \newcommand{\be}{\begin{equation}}
	 \newcommand{\ee}{\end{equation}}
	 \newcommand{\bt}{\beta}
	 \newcommand{\al}{\alpha}
	 \newcommand{\laa}{\lambda_\alpha}
	 \newcommand{\lab}{\lambda_\beta}
	 \newcommand{\no}{|\Omega|}
	 \newcommand{\nd}{|D|}
	 \newcommand{\Om}{\Omega}
	 \newcommand{\h}{H^1_0(\Omega)}
	 \newcommand{\lt}{L^2(\Omega)}
	 \newcommand{\la}{\lambda}
	 \newcommand{\ro}{\varrho}
	 \newcommand{\cd}{\chi_{D}}
	 \newcommand{\cdc}{\chi_{D^c}}
\newcommand{\re}{\Re\mathrm{e}} 
\newcommand{\im}{\Im\mathrm{m}} 
\newcommand{\eps}{\varepsilon}
\newcommand{\z}{\xi}
\newcommand{\s}{\sigma}
\newcommand{\R}{\mathbb{R}}
\newcommand{\N}{\mathbb{N}}
\newcommand{\Z}{\mathbb{Z}}
\renewcommand{\o}{\overline}
\renewcommand{\u}{\underline}

\numberwithin{equation}{section}

\def\debaixodaseta#1#2{\mathrel{}\mathop{\longrightarrow}\limits^{#1}_{#2}}
\def\debaixodasetafraca#1#2{\mathrel{}\mathop{\rightharpoonup}\limits^{#1}_{#2}}
\def\debaixodolim#1#2{\mathrel{}\mathop{\lim}\limits^{#1}_{#2}}
\def\debaixodoinf#1#2{\mathrel{}\mathop{\inf}\limits^{#1}_{#2}}
\def\debaixodoliminf#1#2{\mathrel{}\mathop{\liminf}\limits^{#1}_{#2}}
\def\debaixodolimsup#1#2{\mathrel{}\mathop{\limsup}\limits^{#1}_{#2}}
\def\debaixodosup#1#2{\mathrel{}\mathop{\sup}\limits^{#1}_{#2}}
\def\debaixodomin#1#2{\mathrel{}\mathop{\min}\limits^{#1}_{#2}}

	 \newtheorem{thm}{Theorem}[section]
	 \newtheorem{cor}[thm]{Corollary}	
	 \newtheorem{lem}[thm]{Lemma}
	 \newtheorem{prop}[thm]{Proposition}
	 \theoremstyle{definition}
	 \newtheorem{defn}{Definition}[section]
	 \newtheorem{exam}{Example}[section]
	 \theoremstyle{remark}
	 \newtheorem{rem}{Remark}[section]
	 \numberwithin{equation}{section}
	 \renewcommand{\theequation}{\thesection.\arabic{equation}}
	 \numberwithin{equation}{section}
	 %
	 %
	 %
	\title[Existence Result for Generalized Variational Equality   ]{Existence Result for Generalized Variational Equality  }
	\author[Allahkaram Shafie,  Farid Bozorgnia    ]{Allahkaram Shafie,  Farid Bozorgnia      }
	 \address{Department of Mathematics, Instituto Superior T\'{e}cnico, Lisbon.} \email{bozorg@math.ist.utl.pt}

	 \date{\today}
	
	 \thanks{F. Bozorgnia was  supported by the Portuguese National Science Foundation through FCT fellowships SFRH/BPD/33962/2009}
	
	 \begin{abstract}
	 In this paper we prove the existence of solution to the Stampachia variational inequality under weakened  assumptions  on the given operator. As a consequence,  we provide some sufficient conditions that under them the generalized  equation $0\in T(x)$ has a solution. Furthermore,  by using  generalized  results of continuity and monotonicity, we  extend the related  existence results  and we   answer    an open problem proposed  by  Kassay and Miholka (J Optim Theory Appl 159 (2013) 721-740).

	 \end{abstract}

	 \maketitle
	

	
\[
\]
\[
\]

	  \section{Introduction}

 The  theory of variational inequality has  been investigated extensively  as     methodology  to study of equilibrium problems. Equilibrium is a central concept in numerous disciplines
including economics, management science, operations
research, and engineering, see \cite{GM, FSB, ZLC}.

In  1966,    Hartman and Stampacchia   introduced the  variational inequality   as a tool for the study of
partial differential equations with applications principally drawn from mechanics, see \cite{HS}.

In \cite{KS} existence result for variational inequalities   is given by generalized
monotone operators. As a consequence, the authors  conclude    the subjectivity for  some classes of set-valued operators. By
strengthening the continuity assumptions, they show  similar subjectivity results without any monotonicity assumption.

 Finding  the zeroes of a  set-valued map $T(x)$ are particularly important.    Indeed, zeroes of the
  subdifferential operator of  a function defined  on the same
space are precisely the minimum points of this function. Hence, there is an important
link between the theory of (generalized) monotone operators and optimization theory, see for  instance  \cite{Gk, HH, ASA}.

\section{Preliminarily and     Mathematical background}

Throughout this paper,  $X$ is Banach space, $X^*$  denotes its topological dual and $\langle \cdot ,\cdot\rangle$ the duality pairing. For a nonempty set $A\subset X$, ${\rm cor}A,{\rm cl}A,{\rm cl_w}A,$ and $ {\rm conv}A$, stand for the algebraic interior, closure, weakly closure, and  convex hull of the set $A,$ respectively. Also for $x^*\in X^*$ we denote $\mathbb{R_{++}}x^*=\{tx^*:~~~~~t>0\}$.


 Let us  recall   the classical terminology of generalized monotonicity of set-valued maps that we will use in the sequel. A set valued map $T:X\rightrightarrows X^*$ is said to be
\begin{itemize}
 \item Quasimonotone on a subset $K$, provided that for all $x,y\in K,$
\begin{equation*}
\exists {x^*} \in T(x): {\left\langle {{x^*},y - x} \right\rangle  > 0 \Rightarrow \left\langle {{y^*},y - x} \right\rangle  \ge 0\,} \,\,\,\,\,\forall {y^*} \in T(y);
\end{equation*}

 \item Properly quasimonotone on a subset $K,$ provided that for all
  \[
  \{x_1,x_2,\cdots ,x_n\}\subseteq K, \,  \textrm{  and for all } \,  x\in {\rm conv}\{x_1,x_2,...,x_n\},
   \]
   there exists $i\in \{1,2,\cdots,n\}$ such that
\begin{equation*}
\langle x_i^*,x_i-x\rangle\geq 0~~~\forall x_i^*\in T(x_i);
\end{equation*}

 \item Pseudomonotone on a subset $K,$ provided that for all $x,y\in K,$
\begin{equation*}
\exists {x^*} \in T(x): {\left\langle {{x^*},y - x} \right\rangle  \ge 0 \Rightarrow \left\langle {{y^*},y - x} \right\rangle  \ge 0\,}  \,\,\,\,\forall {y^*} \in T(y).
\end{equation*}

\end{itemize}

 A set-valued operator $T: X\rightrightarrows X^*$ is said to be upper sign-continuous on a convex subset $K,$ if   for any $x,y\in K,$ the following implication holds:
 \begin{equation}
{\forall t \in (0,1)\,\,\,\mathop {\inf }\limits_{x_t^* \in T(x_t)} \left\langle {x_t^*,y - x} \right\rangle  \ge 0\, \Rightarrow \mathop {\sup}\limits_{{x^*} \in T(x)} \left\langle {{x^*},y - x} \right\rangle  \ge 0},
\end{equation}
where $x_t=tx+(1-t)y.$

Accordingly,    $T$ is   called lower sign-continuous on a convex subset $K$ if, for  or any $x,y\in K,$ the following implication holds:
 \begin{equation}
{\forall t \in (0,1)\,\,\,\mathop {\inf }\limits_{x_t^* \in T(x_t)} \left\langle {x_t^*,y - x} \right\rangle  \ge 0\, \Rightarrow \mathop {\inf}\limits_{{x^*} \in T(x)} \left\langle {{x^*},y - x} \right\rangle  \ge 0}.
\end{equation}

By these  definitions it is clear that any lower sign-continuous map is also upper
sign-continuous. Furthermore,   if  $T,S:X\rightrightarrows X^*$ be set-valued maps and $T\subseteq S$ and $T$ be lower sign-continuous, then $S$ is lower sign-continuous.

  By  the following example we underline  that  this implication is not true for lower semi-continuous mappings.
\begin{exam}
Consider the following  set-valued map $T:\mathbb{R\rightrightarrows \mathbb{R}}$ as \[T(x) = \left\{ \begin{array}{l}
 \left\{ { - 1,0} \right\}\,\,\,\,\,\,\,x \ne 0, \\
 \left\{ 0 \right\}\,\,\,\,\,\,\,\,\,\,\,\,\,\,\,\,\,\,\,x = 0. \\
 \end{array} \right.\]
 It is easy to cheek that ${\rm conv}T$ is lower semi-continuous,  but $T$ is not lower semi-continuous.
\end{exam}

Algebraic interior is defined as $${\rm cor} K = \left\{ {y \in K:\,\,\,\forall x \in X,\,\exists \lambda  > 0\,,\,\,\forall t \in [0,\lambda )\,\,\,\,\,\,\,y + t x \in K\,} \right\}.$$
Note  that  always ${\rm int}K\subseteq {\rm cor}K\subseteq K$ and algebraic interior is weaker than of topological interior   by  the following example.
\begin{exam}
Let $X=l^p, p\ge 1$ and   consider the convex set $K$ defined by
$K = \left\{ {x = {{(x(n))}_{n \in\mathbb{N} }} \in {l^p},{\mkern 1mu} {\mkern 1mu} {\mkern 1mu} {\mkern 1mu} {\mkern 1mu} {\mkern 1mu} \forall n \in\mathbb{N} ,{\mkern 1mu} {\mkern 1mu} \,\,\,{x}(n) \ge 0{\mkern 1mu} } \right\}.$ Then   ${\rm int}K=\emptyset$ and $$corK = \left\{ {x = {{(x(n))}_{n \in\mathbb{N} }} \in {l^p},\,\,\,\,   x(n) > 0\,\,\,\,\,\forall n \in\mathbb{N} } \right\}.$$
\end{exam}
Also,  we say that the map $T:X\rightrightarrows X^*$ is locally upper sign-continuous at $x,$  if there exists  a convex neighbourhood $U$ of $x$ and an upper sign-continuous submap $\Phi_x:U\rightrightarrows X^*$ with nonempty convex $w^*$-compact values, satisfying
\[
\Phi_x(u)\subset T(u)\setminus\{0\},\, \textrm{ for any},\, u\in U.
\]
In the sequel, given a set-valued map $T:X\rightrightarrows X^*$ we  consider its
convex hull map ${\rm conv} T:X\rightrightarrows X^*$ defined by ${\rm conv}T(x)={\rm conv}(T(x)).$

The proof of the following  proposition  is straightforward.
\begin{prop}\label{p2}
Let $T:X\rightrightarrows X^*$ be a set-valued map and $x\in {\rm dom}T.$ If $T$ is locally upper sign-continuous at $x,$  then ${\rm conv} T$ is locally upper sign-continuous at $x.$
\end{prop}
The following example shows that the reverse of Proposition \ref{p2} is not true.
\begin{exam}
Let  the set-valued map $T:\mathbb{R}\rightrightarrows\mathbb{R}$ be defined by
\begin{equation}\label{55}
T(x) = \left\{ \begin{array}{l}
 \left\{ { - 1,1} \right\}\,\,\,\,\,\,\,x = 0, \\
 \left[ { - 1,1} \right]\,\,\,\,\,\,\,\,\,\,x \ne 0. \\
 \end{array} \right.
\end{equation}
Note that
\[
{\rm conv}T(x)=[-1,1].
\]
One can check  that ${\rm conv}T$  is locally upper sign-continuous but  $T$ is not locally upper sign-continuous  at $x=0$.
\end{exam}
\begin{defn}
A map $T:X\rightrightarrows X^*$ is said to be weakly dually lower semicontinuous on a subset $K$ if for any $x\in K$ and for any net $(y_{\alpha})_{\alpha}\subseteq K$ such that $y_{\alpha}\rightharpoonup y,$ the following implication holds:
\begin{equation*}
{\mathop {\lim \sup }\limits_{\alpha} \mathop {\inf }\limits_{y_{\alpha}^* \in T({y_{\alpha}})} \left\langle {{y_{\alpha}^*},{y_{\alpha}} - x} \right\rangle\ge 0\,\Rightarrow \mathop {\inf }\limits_{{y^*} \in T(y)} \left\langle {{y^*},y - x} \right\rangle  \ge 0\,}.
\end{equation*}
\end{defn}
It is worth to mention  that any weakly lower semicontinuous map on $K$ is weakly dually lower semicontinuous on $K$ but   this concept is strictly weaker than the lower semicontinuouity. For example choose  $K=[-1,1],$  and  define  the set-valued map $T:\mathbb{R}\rightrightarrows\mathbb{R}$   by
\begin{equation*}
T(x)=\left\{ \begin{array}{l}
 [ - 1,0],\,\,\,\,\,\,\,\,\,\,\,\text{if}~\,x = 0, \\
 \{  0\} ,\,\,\,\,\,\,\,\,\,\,\,\,\,~\text{otherwise}. \\
 \end{array} \right.
\end{equation*}
$T$ is  dually lower semicontinuous on $K$ but it is not lower semicontinuous at $x=0.$ Also note  that if $T$ is  a dually lower semicontinuous,  then ${\rm conv}T$ is so. However,  the reciprocal is not true in general. For instance the set-valued map $T:\mathbb{R}\rightrightarrows\mathbb{R}$ defined by
\begin{equation*}
T(x)=\left\{ \begin{array}{l}
 \mathbb{Q},\,\,\,\,\,\,\,\,\,\,\,\,\text{if}~\,x \ge 0, \\
  \mathbb{Q}^{c},\,\,\,\,\,\,\,\,\,\text{if}~{\kern 1pt} x < 0, \\
 \end{array} \right.
\end{equation*}
is not dually lower semicontinuous but ${\rm conv}T=\mathbb{R}$ is dually lower semicontinuous.

The variational inequality problem which  we consider in this paper can be formulated as follows. Given a nonempty and convex subset $K$ of $X,$ find an element $\bar{x}\in K$ such that
\begin{equation*}
{\mathop {\sup }\limits_{{x^*} \in T(\bar{x})} \left\langle {{x^*},y - \bar{x}} \right\rangle  \ge 0\,\,\,\,\,\,\,\forall y \in K.}~~~{\rm (VI)}
\end{equation*}
We will consider the following concepts of solutions of the Stampacchia variational inequality.
\begin{itemize}
\item Stampacchia solutions:
\begin{equation*}
S(T,K) = \left\{ {x \in K:\,\,\exists {x^*} \in T(x)\,\,\, \text{with} \left\langle {{x^*},y - x} \right\rangle  \ge 0,\,\,\,\forall y \in K} \right\}.
\end{equation*}
\item  Star Stampacchia solutions:
\begin{equation*}
{S^*}(T,K) = \left\{ {x \in K:{\mkern 1mu} {\mkern 1mu} \exists {x^*} \in T(x) \setminus \{ 0\} {\mkern 1mu} ,{\rm{ with}}\,\,\left\langle {{x^*},y - {\rm{ }}x} \right\rangle  \ge 0,{\mkern 1mu} {\mkern 1mu} \,\forall y \in K} \right\}.
\end{equation*}
\item Weak Stampacchia solutions:
\begin{equation}
{S^w}(T,K) = \left\{ {x \in K:{\mkern 1mu} {\mkern 1mu} \forall y \in K\,\,\exists {x^*} \in T(x){\mkern 1mu} {\mkern 1mu} {\rm{ with}}\left\langle {{x^*},y - {\rm{ }}x} \right\rangle  \ge 0 {\mkern 1mu} {\mkern 1mu} {\mkern 1mu} } \right\}.
\end{equation}
\item  Minty solutions
\begin{equation*}
M(T,K) = \left\{ {x \in K:\,\,\forall {y^*} \in T(y)\,,\,\,\forall y \in K,\,\,\left\langle {{y^*},y - x} \right\rangle  \ge 0\,} \right\}.
\end{equation*}
\end{itemize}

We recall that $x\in K$ is a local solution of the Minty variational inequality if there exists a neighborhood $U$ of $x$ such that $x\in M(T,K\cap U).$  The set of all local   solution is denoted  by $LM(T,K)$. It is obvious that $M(T,K)\subseteq LM(T,K).$ The following example illustrates  that converse is not necessarily true.
\begin{exam}
Let $T:\mathbb{R}\rightrightarrows \mathbb{R}$ be set-valued map as \[     T(x) = \left\{ \begin{array}{l}
 \{ 1\} \,\,\,\,\,\,\,x \in ( - 1,1), \\
 \{ 2\} \,\,\,\,\,\,\,\,\,\,\,x \notin ( - 1,1).\\
 \end{array} \right.        \]
  Clearly $0\in LM(T,K)$ but $0\notin M(T,K).$
\end{exam}

\begin{rem}\label{re1}
One  can  see that the solution of  Stampacchia variational inequality problem is  also the solution of the  problem (VI).
\end{rem}

The topic of variational inequality  appears in   the calculus of variations
in minimizing  a functional  over a convex set of constraints. The
  Euler equation must be replaced by a set of inequalities. Here,  we briefly mention the classical obstacle problem.  Consider the following functional, $I (u),$ defined
\[
I (u)=\int_{\Omega} L(x, u, \nabla u)dx.
\]
The Lagrangian $L(x, u, z)$  is assumed to be jointly convex in $(u, z)$, proper, and lower semi-continuous. The  obstacle problem is formulated  as a constrained minimization:
\[
u^*=  \underset{ u\in K} {\textrm{argmin}} \,  I(u);
\]
where  the convex constraint set $K$  is given by
\[
 K={\{ u \in H, u \ge \varphi  \, \, \textrm{in}\, \,  \Omega, \quad u=g \,\, \textrm{on the boundary} }\}.
 \]
  Let $DI$  be the derivative associated with
the G\^{a}teaux differentiable functional $I,$  i.e.
\[
\frac{d}{d\varepsilon}I (u+\varepsilon v)|_{\varepsilon=0}=
\left\langle DI(u), v  \right\rangle.
\]
Then the minimization problem is equivalent to finding  $u^*\in K$ such that:
\[
\left\langle  DI(u^*), u^{*}-v \right\rangle >0, \quad \forall v \in K.
 \]

\section{Existence results}
In this section, we present our results.

\begin{lem}\label{z}
Let $A$ be a subset of $X$ and $x^*\in X^*$ be nonzero and $y\in X$ be given. If $\langle x^{*}, x-y\rangle\geq 0$  for all $x\in A$, then $\langle x^{*},x-y\rangle>0$ for all $x\in {\rm cor}A.$\end{lem}
\begin{proof}
Suppose on the contrary,  there exists $x_0\in {\rm cor}A$ with  $\langle x^{*},x-y\rangle\leq0$ this gives   $\langle x^{*},x-y\rangle=0.$ Consider $z\in X$ so  there is a positive net ${t_{\alpha}\subset\mathbb{R}}$ such that  $x_0+t_{\alpha}z\longrightarrow x_0.$ For  $x_0\in {\rm cor} A$   there exists $\beta$ such that
\[
x+t_{\beta}z\in {\rm cor}A\subseteq A,
\]
 which by assumption implies that $\langle x^{*},x_0+t_{\beta}z-y\rangle\geq0.$ From the last relation, we obtain
\begin{equation*}
 \langle {x^*},{x_0} + {t_\beta }z - y \rangle  =  \langle {x^*},{x_0} - y \rangle  + {t_\beta } \langle {x^*},z \rangle  = {t_\beta } \langle {x^*},z \rangle  \ge 0,
\end{equation*}
hence $ \langle{x^*},z \rangle\ge 0.$  Next, since $z\in X$ is    arbitrary, we conclude  that  $x^*=0$       which is contradiction.
\end{proof}
Notice that if  $A$ is convex set,  then the reverse of Lemma \ref{z}  holds.

We need the following lemma in the sequel.
\begin{lem}\label{l1}
Let $A$ be a convex subset of $X$ and ${\rm cor} A\ne\emptyset.$  Then
\[
{{\rm cl_w}A={\rm cl_w}({\rm cor}A).}
\]
\end{lem}
\begin{proof}
Clearly, we have  ${\rm cl_w}({\rm cor}A)\subseteq {\rm cl_w}A.$ To see the  reverse inclusion,
   let $x\in{\rm cl_w}A, a\in{\rm cor}A,$ then there exists $x_{\alpha}\in A$ such that $x_{\alpha}\rightharpoonup x.$     Thanks to Lemma 1.9 in  \cite{JJ},  one has $[a,x_{\alpha})\subset{\rm cor}A.$ Next  choose $0<t_{\alpha}<1$ such that $t_{\alpha}\longrightarrow 0.$ Hence $t_{\alpha}a+(1-t_{\alpha})x_{\alpha}\in corA.$ On the other hand
     \[
     t_{\alpha}a+(1-t_{\alpha})x_{\alpha} \rightharpoonup x,
      \]
      which implies that $x\in {\rm cl_w}({\rm cor}A).$
\end{proof}
In the next Lemma,  we  provide conditions on  the map $T$ that  relate the $LM$ solutions and  $S^{w}$ solutions.

\begin{lem}\label{l2}
Let $K$ be nonempty convex subset of  $X$  and  $T:X\rightrightarrows X^*$ be a set-valued map.
If ${\rm conv}T$ is locally upper sign-continuous, then
\[
LM(T,K)\subseteq S^w(T,K).
\]
\end{lem}
\begin{proof}
First   assume that $x\in LM(T,K),$  then  there exists a convex neighborhood $U_x$ of $x$ such that $x\in M(T,K\cap U_x).$ On the other hand,  by locally upper sign-continuity of ${\rm conv}T$ there exists a convex neighborhood $V_x$ of $x,$ and an upper sign-continuous submap $\Phi_x:V_x\rightrightarrows X^*$ with non-empty convex $w^*$-compact values satisfying
\[
\Phi_x(v)\subseteq {\rm conv}T(v)\setminus\{0\},  \,  \textrm{  for any  } \, v\in V_x.
\]
 Hence, $x\in M(T,K\cap U_x\cap V_x).$ Now, let $y$ be an element of $K,$ since $K \cap U_x\cap V_x $ is convex then   there exists $y_1\in [x,y]\cap U_x\cap V_x$ such that
 \[
 [x,y_1]\subseteq K\cap U_x\cap V_x.
  \]
Thus  one has
\[
\left\langle {{z^*},z - x} \right\rangle  \ge 0,\, \, \textrm{  for all }  \,\, z\in [x,y_1]\, \textrm{  and } \, z^*\in\Phi_x(z).
\]
 Hence
\begin{equation*}
\mathop {\inf }\limits_{z \in [x,{y_1}]} \mathop {\inf }\limits_{{z^*} \in {\Phi _x}(z)} \left\langle {{z^*},z - x} \right\rangle  \ge 0.
\end{equation*}
Upper sign-continuity of $\Phi_x$ implies that $$\mathop {\sup }\limits_{{x^*} \in {\Phi _x}(x)} \left\langle {{x^*},z - x} \right\rangle  \ge 0,\,\,\,\,\,\,\,\forall z \in [x,{y_1}].$$
Now, for $z=y_1$ one can obtain $$\mathop {\sup }\limits_{{x^*} \in {\Phi _x}(x)} \left\langle {{x^*},{y_1} - x} \right\rangle  \ge 0.$$  Since $\Phi_x(x)$ is compact, there exists $x^*_y\in\Phi_x(x)\subseteq {\rm conv}T(x)\setminus\{0\}$ such that
\[
\left\langle {{x_y^*},{y_1} - x} \right\rangle  \ge 0.
\]
On the other hand,  there exists $0<t<1$  such that $y_1=tx+(1-t)y$ and therefore $\left\langle {{x_y^*},y - x} \right\rangle  \ge 0.$ Now since $x^*_y\in {\rm conv}T(x),$  then  there exists   $0\le t_i\le 1$ such that
\[
x_y^* = \sum\limits_{i = 1}^n {{t_i}x_{iy}^*} ,\,\,\,\,\,\sum\limits_{i = 1}^n {{t_i} = 1,\,\,\,} \,\,x_{iy}^* \in T(x).
\]
This implies
 $$\left\langle {\sum\limits_{i = 1}^n {{t_i}x_{iy}^*} ,y - x} \right\rangle  = \sum\limits_{i = 1}^n {{t_i}\left\langle {x_{iy}^*,y - x} \right\rangle }  \ge 0.$$
      Therefore there exists $0\le j\le 1$ such that $\left\langle {x_{jy}^*,y - x} \right\rangle \ge 0$ and so one has $x\in S^{w}(T,K).$
\end{proof}

\begin{prop}\cite{AN}\label{p1}
Let $K$ be a nonempty, convex subset of the topological vector space $X$ and let $T:X\rightrightarrows X^*$ be quasimonotone  and is not properly quasimonotone. Then one has $LM(T,K)\ne\emptyset.$
\end{prop}
We need the following Lemma in the sequel.
\begin{lem}\label{s}
Let $K$ be a weakly compact subset of $X.$ If $T:X\rightrightarrows X^*$ is quasimonotone, then $LM(T,K)\ne\emptyset.$
\end{lem}
\begin{proof}
The proof is straightforward by Theorem 5.1  in  \cite{DA} and Proposition \ref{p1}.
\end{proof}
The following Theorem is an  extension of Theorem 3.1 in \cite{Gk} without   coercivity, locally bounded and hemiclosed conditions  on $T$ and reflexivity of Banach space $X.$
\begin{thm}\label{th1}
Let $K$ be a nonempty convex subset of $X.$ Assume that $T:X\rightrightarrows X^*$ be a quasimonotone operator that is not properly quasimonotone. If ${\rm conv}T$ is locally upper sign-continuous, then  the variational inequality ${\rm (VI)}$ has a solution. If moreover, $K=X$ and for all $x\in K,$ $T(x)$ is weakly compact, then the generalized
equation $0 \in T(x)$ admits a solution.
\end{thm}
\begin{proof}
By Lemma \ref{l2} and Proposition \ref{p1} and Remark \ref{re1} it is easy    to cheek   the existence of solution
 ${\rm (VI)}.$ Now, let $\bar{x}$ be a solution of variational inequality ${\rm (VI)}.$ Since $T(\bar{x})$ is weakly compact for $y\in K$ there exists $x^*_y\in T(\bar{x})$ such that
 \[
 0 \le \mathop {\sup }\limits_{{x^*} \in T(\overline x )} {\mkern 1mu} {\mkern 1mu} {\mkern 1mu} \left\langle {{x^*},y - {\rm{ }}\overline x } \right\rangle  = \left\langle {x_y^*,y - {\rm{ }}\overline x } \right\rangle.
  \]
  This means that $\{0\}$ cannot be strongly separated from the closed convex set $T(x)$
and therefore, $0 \in T(x).$
\end{proof}
In the case that  operator $T$ is  properly quasimonotone, we can not use  the Proposition \ref{th1}.  In order to overcome this flaw, under the weaker condition of Theorem 2.1 of \cite{AN} one can get the following result without any coercivity  condition.
\begin{thm}
Let $K,U$ be  nonempty convex subsets of $X$ and $K\cap {\rm cor}U$ be nonempty and weakly compact. Further, let $T:X\rightrightarrows X^*$ be a quasimonotone operator on $K.$  If $T$ is locally upper
sign-continuous, then  the variational inequality ${\rm (VI)}$ has a solution.
\end{thm}
\begin{proof}
Suppose that $T$ be properly quasimonotone,  hence by Lemma \ref{s} one has $LM(T,K\cap{\rm cor}U)\ne\emptyset.$ Choose $x_0\in LM(T,K\cap{\rm cor}U),$  then

\[\exists \,  x^{*}_0 \in T({x^{}}_0),\,\,\,\,\,\forall y \in K\cap {corU} ;\,\,\,\,\,\,\,\left\langle {x^{*}_0,y - {x^{}}_0} \right\rangle  \ge 0.\]
Now for every $z\in K$ there exists $t>0$ such that
\[
 x_0+t (z-x_0)\in K\cap {\rm cor}U,
 \]
 which implies that $\left\langle {x^{*}_{0},z - {x^{}}_0} \right\rangle  \ge 0.$  Therefore $x_0\in S(T,K)$ which completes the proof.
\end{proof}
\begin{lem}
Let $K$ be a convex subset of $X$ with ${\rm cor}K\ne\emptyset$ and  the set valued map  $T:X\rightrightarrows X^*$ be  quasimonotone and weakly dually lower semicontinuous on $K.$ If $S(T,K)\nsubseteq M(T,K),$ then the generalized
equation $0 \in T(x)$ admits a solution.

\end{lem}
\begin{proof}
Suppose that for all $x\in X$ we have $0\notin T(x),$ and
 $x\in S(T,K)$ be given. Hence,  there exists $x^*\in T(x)$  such that
  \[
  \langle {x^*},y - x \rangle \ge 0,\, \textrm{  for all} \, y\in K, \, \textrm{ and } \, x^*\ne 0.
   \]
   By Lemma \ref{z} it follows that $$\langle {x^*},z - x\rangle  > 0,\,\,\,\,\,\,\,\,\,\,\,\forall z \in {\rm cor}K.$$ For  any $y\in K$ by Lemma \ref{l1} there exists net $y_{\alpha}\in {\rm cor}K$ such that $y_{\alpha}\rightharpoonup y.$ Consequently,  for any $\alpha,$ $\langle {x^*},y_{\alpha} - x \rangle > 0$ and thus by quasimonotonicity,
    \[
    \langle {y_{\alpha}^*},y_{\alpha} - x \rangle \geq 0, \, \, \textrm{ for all}\, \, {y_{\alpha}^*}\in T(y_{\alpha}).
    \]
     Finally,  by weakly dually lower semicontinuity at $T$ one has $\langle y^*,y - x \rangle \geq 0$ for each $y^*\in T(y).$ The later indicates that  $x\in M(T,K),$ therefore
     \[
     S(T,K)\nsubseteq M(T,K).
     \]

\end{proof}
\begin{rem}
It is worth to note that the condition $(D)$ or (4.1)  in \cite{HH}  on the set $K\subseteq X$   is equivalent  to $M(T,K)\ne\emptyset.$ Also if ${\rm int}K\ne\emptyset$ in condition $D$ then $LM(T,K)\ne\emptyset.$
\end{rem}
The proof of the following Propositions (\ref{s1}) and (\ref{s2})  are straightforward.
\begin{prop}\label{s1}
Assume that $T:X\rightrightarrows X^*$ is an upper sign-continuous set-valued on $K$ whose values are convex and compact sets. If $M(T,K)\ne\emptyset,$ then the variational inequality $(VI)$ has a solution.
\end{prop}
\begin{prop}\label{s2}
Assume that $T:X\rightrightarrows X^*$ is a  locally upper sign-continuous set-valued on $K$ whose values are convex and $w^*$-compact sets. If $M(T,K)\ne\emptyset,$ then the generalization $0\in T(x)$ has a solution.
\end{prop}

\begin{lem}
Let $K$ be  an  algebraically open set in $X.$ Then one has $ S(T,K) \subseteq Z_T. $ Here  $Z_T$ is the set all zeros of $T,$ {\rm i.e} $Z_T=\{x\in X:~~~0\in T(x)\}.$
\end{lem}
\begin{proof}
Suppose that $\bar{x}\in S(T,K),$ then there exists $x^*\in  T(\bar{x})$ such that for all $y\in K$ one has $\langle x^*,y-\bar{x}\rangle\ge 0$. Since $\bar{x}\in {\rm cor} K$ hence for given $x\in X$ one have $\bar{x}+tx\in K$ for some $t>0.$ So we have $\langle x^*,(\bar{x}+tx)-\bar{x}\rangle\ge 0$ which implies that $\langle x^*,x\rangle\ge 0$ and this means that $x^*=0$ thus $0\in T(\bar{x}).$
\end{proof}
\begin{prop}
Let $T$ be a quasimonotone operator, which ${\rm conv}T$  is lower semi-continuous at $x\in{\rm cor}K.$ Then

 \[(\forall x^*\in X^*\setminus\{0\},~~~T(x)\not  \subseteq  \mathbb{R_{++}}x^*)\Longleftrightarrow 0\in T(x).\]
\end{prop}
\begin{proof}
Suppose in  the contrary, $0\notin T(x)$ then there exists $x^*\in T(x)$ and $y^*\in T(x)\setminus\{0\}$ such that for every $\lambda>0,$ one has $x^*\ne \lambda  y^*.$ Hence there exists $w\in X$ such that $$\langle x^*,w\rangle > 0>\langle y^*,w\rangle.$$ Obviously $x^*\in {\rm conv} T(x)$  and $x+\frac{1}{n}w\longrightarrow x.$  By using  the lower semi-continuity of ${\rm conv}T,$   there exists $y_n^*\in {\rm conv}T(x+\frac{1}{n}w)$ such that $y_n^*\longrightarrow y^*.$  On the other hand, one can have
\[
\langle x^*,(x+\frac{1}{n}w)-x\rangle>0.
\]
Since $T$ is quasimonotone,then  for $y_n^*\in {\rm conv}T(x+\frac{1}{n}w)$ it  holds
\[
\langle y_n^*,(x+\frac{1}{n}w)-x\rangle\ge 0,
\]
 and thus $\langle y^*,w\rangle\ge 0,$ which is contradiction.
\end{proof}
By similar argument in previous proposition  one can prove the following.
\begin{prop}\label{q}
Let $T$ be a quasimonotone operator, which is lower sin continuous on $K,$ and $x\in{\rm cor}K.$ Then \[(\forall x^*\in X^*\setminus\{0\},~~~T(x)\not  \subseteq  \mathbb{R_{++}}x^*)\Longleftrightarrow 0\in T(x).\]
\end{prop}


\begin{thebibliography}{99}

\bibitem{HS} P.  Hartman,  G. Stampacchia, {\em On some non-linear elliptic differential-functional equations.}      Acta Mathematica, 115 (1966),  271--310.\\


\bibitem{KS}  D.  Kinderlehrer,   G.  Stampacchia, {\em  An Introduction to Variational Inequalities and Their Applications.}   Academic Press, New York, 1980.\\




\bibitem{AN} D. Aussel, N. Hajisavvas, {\em On Quasimonotone Variational Inequalities}
J.Optim.Theory Appl. 2,  (2004),445-450.\\


\bibitem{Gk}  G.  Kassay, M.  miholca, {\em Existence results for variational inequalities with surjectivity consequnces related to generalized monotone oparator.}  J.Optim. Theory Appl. 159, (2013), 721-740 .\\

\bibitem{GM} F. Giannessi,  A. Maugeri,  {\em  Variational Inequalities and Network Equilibrium Problems.} Springer,  1995. \\


\bibitem{JJ} J.  Johannes  {\em Vector Optimization Theory and Application.} Springer, 2010.\\

\bibitem{HH}  H. Hassouni , {\em Quasimonotone multifunctions applications to optimality conditions in quasiconvex programing.} J.Num.Fun.Anal and Appl.(1992), 67-275 .\\

\bibitem{FSB}   S. Z. Fatemia, M. Shamsi and F. Bozorgnia,  {\em Extragradient Methods for Differential Variational Inequality and Linear Complementarity Systems.}  Math Meth Appl Sci.  40, (2017)   7201–7217.\\

\bibitem{DA}  A.  Daniilids,  and N. Hadjsavvas,  {\em Characterization of Nonsmooth Semistrictly Quasiconvex and Strictly Quasiconvex Functions.} J.Optim.Theory Appl.  Vol. 102,  (1999),   525-536.\\

\bibitem{YH} H.  Yiran,  {\em Solvability of the Minty variational inequality.} J.Optim.Theory Appl.  Vol. 102, (2017)   525-536.\\


\bibitem{ZLC} S. S. Zhang,   J. H W. Lee, and K. Chan,
{\em Algorithms of common solutions to quasi variational inclusion
and fixed point problems.} Appl. Math. Mech.   29(5),  (2008), 571-581.\\


\bibitem{ASA} A. Shafie, S.J afari, and  A. Farajzadeh,
{\em Existence result for minty variational inequalities with surjectivity consequences in hausdorff topological vector spaces.} J.Nonlinear.convex.Analysis.   18, (2017), 685-696.




\end{thebibliography}
\end{document}